\documentclass[10pt]{amsart}
\usepackage{latexsym}
\usepackage{amscd}
\usepackage{amssymb}
\usepackage[all]{xy}
\usepackage{amsmath, pb-diagram}
\DeclareMathAlphabet{\eusm}{OT1}{eusm}{m}{n}
\newtheorem{theor}{Theorem}[section]
\newtheorem{prop}[theor]{Proposition}
\newtheorem{defi}[theor]{Definition}
\newtheorem{cor}[theor]{Corollary}
\newtheorem{lem}[theor]{Lemma}
\newtheorem{notation}[theor]{Notation}
\newtheorem{exam}[theor]{Example}

\newtheorem{remarks}[theor]{Remarks}

\def\Ker{\mbox{Ker\/}}
\def\Im{\mbox{Im\/}}
\def\End{\mbox{End\/}}
\def\Aut{\mbox{Aut\/}}
\def\Gal{\mbox{Gal\/}}
\def\coGal{\mbox{coGal\/}}

\def\H{\mbox{Hom\/}}
\def\M{\mbox{Mod\/}}

\begin{document}
\title[Modules invariant under automorphisms]{Modules invariant under automorphisms of their covers and envelopes}
\author{Pedro A. Guil Asensio}
\thanks{The first author was partially supported by the DGI (MTM2010-20940-C02-02)
and by the Excellence Research Groups Program of the S\'eneca Foundation of the Region of Murcia. Part of
the sources of both institutions come from the FEDER funds of the European Union.}
\address{Departamento de Mathematicas, Universidad de Murcia, Murcia, 30100, Spain}
\email{paguil@um.es}
\author{Derya Kesk\.{i}n T\"ut\"unc\"u}
\address{Department of Mathematics, Hacettepe University, Ankara, 06800, Turkey}
\email{keskin@hacettepe.edu.tr}
\author{Ashish K. Srivastava}
\address{Department of Mathematics and Computer Science, St. Louis University, St.
Louis, MO-63103, USA}
\email{asrivas3@slu.edu}

\begin{abstract}
In this paper we develop a general theory of modules which are invariant under automorphisms of their covers and envelopes. When applied to specific cases like injective envelopes, pure-injective envelopes, cotorsion envelopes, projective covers, or flat covers, these results extend and provide a much more succinct and clear proofs for various results existing in the literature. Our results are based on several key observations on the additive unit structure of von Neumann regular rings.
\end{abstract}

\maketitle

\section{Introduction.}

\noindent The study of modules which are invariant under the
action of certain subsets of the endomorphism ring of their
injective envelope can be drawn back to the pioneering work of
Johnson and Wong \cite{JW} in which they characterized
quasi-injective modules as those modules which are invariant under
any endomorphism of their injective envelope. Later, Dickson and
Fuller \cite{DF} initiated the study of modules which are
invariant under the group of all automorphisms of their injective envelope. Such modules are called {\it automorphism-invariant
modules}. Fuller and Dickson proved that any indecomposable
automorphism-invariant module over a $K$-algebra $A$ is
quasi-injective provided that $K$ is a field with more than two
elements. And this result has been recently extended in \cite{GS1}
to arbitrary automorphism-invariant modules $M$ such that their
endomorphism ring has no homomorphic images isomorphic to the
field of two elements $\mathbb{Z}_2$.

Although, in the general setting, an
automorphism-invariant module does not need to be quasi-injective (see Example \ref{no-qi}),
its endomorphism ring  shares several important properties with
the endomorphism ring of a quasi-injective module. For example, it has been proved in
\cite{GS2} that the endomorphism ring of an
automorphism-invariant module $M$ is always a von Neumann regular
ring modulo its Jacobson radical $J$, idempotents lift modulo $J$,
and $J$ consists of those endomorphisms of $M$ which have essential
kernel. Analogous results were obtained by
Warfield \cite{Warfield1} for injective modules; by Faith and Utumi \cite{FU} for the case of quasi-injective modules; and by Huisgen-Zimmermann and Zimmermann for pure-injective modules \cite{ZZ}. Moreover, it
has been shown in \cite{ESS}  that any automorphism-invariant module $M$
is of the form $M=A\oplus B$ where $A$ is quasi-injective and $B$
is square-free. As a consequence of these results, it follows that
any automorphism-invariant module satisfies the full exchange
property which extends results of Warfield \cite{Warfield1} and Fuchs \cite{Fuchs}.
These modules also provide a new class of clean
modules (see \cite{GS2}). A dual notion of automorphism-invariant
module has been recently introduced in
\cite{SS1}.

The objective of this paper is to provide a general setting where
the above results can be obtained by developing  a general theory
of modules which are invariant under automorphisms of their covers
and envelopes. We define the notions of $\mathcal
X$-automorphism-invariant and $\mathcal
X$-automorphism-coinvariant modules, where $\mathcal X$ is any
class of modules closed under isomorphisms. And we show that the
above characterizations are particular instances of much more
general results for these modules. Namely, we prove that if  $u:
M\rightarrow X$ is a monomorphic $\mathcal{X}$-envelope of a module
$M$ such that $M$ is $\mathcal{X}$-automorphism-invariant,
$\End(X)/J(\End(X))$ is a von Neumann regular right self-injective
ring and idempotents lift modulo $J(\End(X))$, then
$\End(M)/J(\End(M))$ is also von Neumann regular and idempotents
lift modulo $J(\End(M))$ and consequently, $M$ satisfies the finite exchange property. Moreover if we assume that every direct summand of $M$ has an $\mathcal X$-envelope, then in this case, $M$ has a
decomposition $M=A\oplus B$ where $A$ is square-free and $B$ is
$\mathcal X$-endomorphism-invariant. If we assume in addition that for $\mathcal X$-endomorphism-invariant modules, the finite exchange property implies the full exchange property, then $M$ also satisfies the full exchange property,  thus extending results of Warfield \cite{Warfield1}, Fuchs
\cite{Fuchs} and Huisgen-Zimmermann and Zimmermann \cite{ZZ1} for injective, quasi-injective and pure-injective modules, respectively. As a consequence of our results, it also follows that $\mathcal{X}$-automorphism-invariant modules $M$, whose every direct summand of $M$ has an $\mathcal X$-envelope, are clean. We also dualize these
results and obtain similar results for $\mathcal
X$-automorphism-coinvariant modules.

Let us note that the fact that automorphism-invariant and automorphism-coinvariant modules inherit these good properties from the endomorphism rings of their envelopes and covers is not surprising. As pointed out in Remarks \ref{galois} and \ref{dualgalois}, a module $M$ having a monomorphic $\mathcal X$-envelope $u:M\rightarrow X$ is $\mathcal X$-automorphism-invariant if and only if $u$ establishes an isomorphism of groups $\Aut(M)\cong \Aut(X)/\Gal(u)$, where $\Gal(u)$ is the Galois group of the envelope $u$ (see e.g. \cite{EGJO,EEG}). Whereas a module $M$ having an epimorphic $\mathcal X$-cover $p:X\rightarrow M$ is $\mathcal X$-automorphism-coinvariant if and only if $p$ induces an isomorphism $\Aut(M)\cong \Aut(X)/\coGal(p)$, where $\coGal(p)$ is the co-Galois group of the cover $p$.

The crucial steps in proving these results are the
observations made in Section 2 about the additive unit structure of a von Neumann regular ring. In this section, we prove that if $S$ is a right self-injective von Neumann
regular ring and $R$ is a subring of $S$ which is stable
under left multiplication by units of $S$, then $R$ is von Neumann
regular and moreover, $R=R_1 \times R_2$, where $R_1$ is a
an abelian regular ring and $R_2$ is a von Neumann
regular right self-injective ring which is invariant under left
multiplication by elements in $S$. These observations are based on
the classification theory developed by Kaplansky for Baer rings \cite{Kap2}. Moreover, as a consequence of these, we deduce that if $S$ is a right self-injective von Neumann regular ring of characteristic $n\geq 0$, then the image $S'$ of the group ring $\mathbb{Z}_n[\Aut(S)]$ inside $S$ under the homomorphism sending an element of $\Aut(S)$ to the corresponding element in $S$ is also von Neumann regular, where we denote by $\Aut(S)$ the group of units of $S$. And $S=S'$ provided that $S$ has no homomorphic images isomorphic to $\mathbb{Z}_2$; in particular, when ${\rm char}(S)=n>0$ and $2\nmid n$ (see Corollaries \ref{reg1}, \ref{reg2}).

These results are finally applied in Section 5 to a wide variety of classes of modules such as the classes of injective modules, projective modules, pure-injective modules, cotorsion modules, flat modules and many interesting results are obtained which extend, clarify and simplify the proofs of the results in \cite{DF,ESS,FU,Fuchs,GS2,SS2,Warfield1,ZZ,ZZ1}.

Throughout this paper, $R$ will always denote an associative ring with identity element and modules will be unital right modules unless otherwise is stated. $J(R)$ will denote the Jacobson radical of the ring $R$. We refer to \cite{AF} and \cite{mm} for any undefined notion arising in the text.

\bigskip

\section{Observations on the additive unit structure of a von Neumann regular ring.}

\noindent In this section, we will establish some useful tools dealing with additive unit representation of von Neumann regular rings. The study of the additive unit structure of rings has its roots in the investigations of Dieudonn\'{e} on Galois theory of
simple and semisimple rings \cite{D}. Laszlo Fuchs raised the question of determining when an endomorphism ring is generated additively by automorphisms, which has been studied by many authors (see, for example \cite{H,St}). Let us note that the question of characterizing when $\mathcal X$-automorphism-invariant modules are $\mathcal X$-endomorphism-invariant is linked to the above question of Fuchs.

We say that an $n\times n$ matrix $A$ over a ring $R$ \textit{admits a diagonal reduction} if there exist invertible matrices $P,Q\in \mathbb M_{n}(R)$ such that $PAQ $ is a diagonal matrix. Following Ara et al. \cite{AGMP}, a ring $R$ is called an \textit{elementary divisor ring} if every square matrix over $R$ admits a diagonal reduction. This definition is less stringent than the one proposed by Kaplansky in \cite{Kap}. The class of elementary divisor rings includes right self-injective von Neumann regular rings.

It is not difficult to see that if $R$ is any ring, then any $n\times n$ (where $n\ge 2$) diagonal matrix over $R$ is the sum of two invertible matrices (see \cite{Henriksen1}). Thus it follows that if $R$ is an elementary divisor ring, then each element in the matrix ring $\mathbb M_n(R)$ is the sum of two units for $n\ge 2$.

 The next lemma is inherent in \cite[Theorem 1]{KS1}. We provide the proof for the sake of the completeness as the result is not stated in this form there.

\begin{lem} \label{type}
If $S$ is a right self-injective von Neumann regular ring, then $S=T_1\times T_2$, where $T_1$ is an abelian regular self-injective ring and any element in $T_2$ is the
sum of two units.
\end{lem}

\begin{proof}
Let $S$ be a right self-injective von Neumann regular ring. Using the Type theory due to Kaplansky, we know that $S= R_1 \times R_2 \times R_3 \times R_4 \times R_5$ where $R_1$ is of type $I_f$, $R_2$ is of type $I_\infty$, $R_3$ is of type $II_f$, $R_4$ is of type $II_\infty$, and $R_5$ is of type $III$ (see \cite[Theorem 10.22]{Goodearl}).  Taking $P=R_2 \times R_4 \times R_5$, we may write $S= R_1 \times R_3 \times P$, where $P$ is purely infinite. We have $P_P\cong nP_P$ for all positive integers $n$ by \cite[Theorem 10.16]{Goodearl}. In particular, for $n=2$, this yields $P \cong \mathbb M_{2}(P)$. Since $P$ is an elementary divisor ring, it follows that each element of $\mathbb M_{2}(P)$ and consequently, each element of $P$ is the sum of two units. Since $R_3$ is of type $II_f$, we can write $R_3 \cong n(e_{n}R_3)$ for each $n \in \mathbb N$ where $e_n$ is an idempotent in $R$ (see \cite[Proposition 10.28]{Goodearl}). In particular, for $n = 2$ we have $R_3 \cong \mathbb M_{2}(e_{3}R_{3}e_{3}$). As $e_{3}R_{3}e_{3}$ is an elementary divisor ring, it follows that each element in $\mathbb M_{2}(e_{3}R_{3}e_{3})$ and hence each element in $R_3$ is the sum of two units. We know that $R_1 \cong \prod\mathbb{M}_{n_i}(A_i)$ where each $A_i$ is an abelian regular self-injective ring (see \cite[Theorem 10.24]{Goodearl}).  Since each $A_i$ is an elementary divisor ring, we know that each element in $M_{n_i}(A_i)$ is the sum of two units whenever $n_i \ge 2$.

Thus we conclude that $S=T_1\times
T_2$, where $T_1$ is an abelian regular right self-injective ring and any element in $T_2$ is the
sum of two units.
\end{proof}

\begin{notation}
\rm We know by the above lemma that any right self-injective von Neumann regular ring $S$ can be decomposed as $S=T_1\times T_2$, where $T_1$ is an abelian  regular ring and any element in $T_2$ is the sum of two units. Therefore, any right $S$-module $N$ is of the form
$N=N_1\times N_2$  where $N_1$ is a right $T_1$-module and $N_2$ a
right $T_2$-module. And any homomorphism $f: N=N_1\times
N_2\rightarrow M=M_1\times M_2$ is of the form $f=f_1\times
f_2$, where $f_1: N_1\rightarrow M_1$ is a $T_1$-homomorphism
and $f_2: N_2\rightarrow M_2$ is a $T_2$-homomorphism.

We will follow this notation along the rest of this section.
\end{notation}

\begin{lem} \label{key1} Let $u_1\times u_2: N_1\times N_2\hookrightarrow T_1\times
T_2$ be a right ideal and $f_1\times f_2: N_1\times
N_2\rightarrow T_1\times T_2$, a homomorphism. If $f_1$ is
monic, then $f_1\times f_2$ extends to an endomorphism $\varphi_1\times
\varphi_2$ of $T_1\times T_2$ which is the sum of three units.
\end{lem}

\begin{proof} Let $K_1$ be a complement of $N_1$ in $T_1$. Then $T_1=E(K_1)\oplus
E(N_1)$ and $f_1$ extends to an injective homomorphism $g_1:
E(N_1)\rightarrow T_1$. As $T_1$ is von Neumann regular and right
self-injective, $\Im g_1$ is a direct summand of $T_1$. Write
$T_1=\Im g_1\oplus L_1$. And, as $T_1$ is directly-finite being an
abelian regular ring and $E(N_1)\cong \Im g_1$, we deduce that
there exists an isomorphism $h_1: E(K_1)\rightarrow L_1$. Let now
$\varphi_1: E(N_1)\oplus E(K_1)\rightarrow \Im g_1\oplus L_1$ be
the automorphism induced by $(g_1, h_1)$. Clearly, $\varphi_1\circ
u_1=f_1$. On the other hand, as $T_2$ is right self-injective,
there exists a $\varphi_2: T_2\rightarrow T_2$ such that
$\varphi_2\circ u_2=f_2$. Write
$\varphi_2=\varphi_2^1+\varphi_2^2$ as the sum of two units in
$T_2$. Again, $\varphi_2^2=\varphi_2^3+\varphi_2^4$ is the sum of
two units in $T_2$.

Now it is clear by construction that $(\varphi_1 \times
\varphi_2)\circ (u_1 \times u_2)=f_1\times f_2$. Moreover,
$$
(\varphi_1 \times \varphi_2)=(\varphi_1\times
\varphi_2^1)+(0\times \varphi_2^2)=(\varphi_1\times \varphi_2^1)+
(1_{T_1}\times \varphi_2^3)+(-1_{T_1}\times \varphi_2^4)
$$
is the sum of three units in $T_1\times T_2$.
\end{proof}

As a consequence of these observations on the additive unit representation of $S$, we can prove the following useful result.

\begin{theor}\label{key2} Let $S$ be a right self-injective
von Neumann regular ring and $R$, a subring of $S$. Assume that $R$ is stable
under left multiplication by units of $S$. Then $R$ is also von Neumann regular.

Moreover, $R=R_1 \times R_2$, where $R_1$ is an abelian regular ring and $R_2$ is a von
Neumann regular right self-injective ring which is invariant under
left multiplication by elements in $S$.
\end{theor}

\begin{proof} By Lemma \ref{type}, we have $S=T_1\times T_2$ where $T_1$ is an abelian regular self-injective ring and each element in $T_2$ is the sum of two units. As $R$ is a subring of $S$, we may write $R=R_1\times R_2$ where $R_1$ is a subring of $T_1$ and $R_2$ a subring of $T_2$. 

The assumption gives that all units of $S$ are also in $R$. Let us choose any $t_2\in T_2$.
We know that $t_2=\phi+\psi$, where $\phi,\psi$ are units in
$T_2$. Therefore, $1_{T_1}\times \phi$ and $1_{T_1}\times \psi$
are units in $S$ and thus, $(1_{T_1}\times \phi)\circ
(1_{R_1}\times 1_{R_2})\in R$ and $(1_{T_1}\times \psi)\circ
(1_{R_1}\times 1_{R_2})\in R$ by hypothesis. And this means that
$\phi\circ 1_{R_2}\in R_2$ and $\psi\circ 1_{R_2}\in R_2$. Thus,
$t_2=t_2\circ 1_{R_2}=(\phi\circ 1_{R_2})+(\psi\circ 1_{R_2})\in
R_2$. This shows that $T_2\subseteq R_2$ and hence $T_2=R_2$. In particular, this shows that $T_2\subseteq R$ and so $T_2$ is a von Neumann regular ideal of $R$. Since every abelian regular ring is unit regular \cite[Corollary 4.2]{Goodearl}, if we have any $x\in T_1$, there exists a unit $u\in T_1$ such that $x=xux$. Further, $u+1_{T_2}$ is a unit of $S$, so it is in $R$. This shows
that $R/T_2$ is von Neumann regular. Since $T_2$ is a von Neumann regular ideal and $R/T_2$ is a von Neumann regular ring, it follows from \cite[Lemma 1.3]{Goodearl} that $R$ is also a von Neumann regular ring.

Now, as $R=R_1\times R_2$, clearly both $R_1$ and $R_2$ are von Neumann regular rings. Since every idempotent of $R_1$ is an idempotent of $T_1$ and $T_1$ is abelian regular, every idempotent of $R_1$ is in the center of $T_1$ and consequently, in the center of $R_1$. Thus $R_1$ is an abelian regular ring. In the first paragraph we have seen that $R_2=T_2$. Thus $R_2$ is a von
Neumann regular right self-injective ring which is invariant under
left multiplication by elements in $S$.
\end{proof}

We are going to close this section by pointing out several interesting consequences of the above theorem.

\begin{prop}\label{no-F2}
Let $S$ be a right self-injective
von Neumann regular ring and $R$, a subring of $S$. Assume that $R$ is stable
under left multiplication by units of $S$. If $R$ has no homomorphic images isomorphic to $\mathbb{Z}_2$, then $R=S$.
\end{prop}

\begin{proof}
Note that our hypothesis implies that $S$ has no homomorphic image
isomorphic to $\mathbb{Z}_2$ since
otherwise, if $\psi:S\rightarrow \mathbb{Z}_2$ is a ring
homomorphism, then $\psi|_R:R\rightarrow \mathbb{Z}_2$ would give
a ring homomorphism, contradicting our assumption. Therefore, each
element in $S$ is the sum of two units by \cite{KS1} and thus, $R$
is invariant under left multiplication by elements of $S$. But
then, calling $1_R$ to the identity in $R$, we get that $s=s\cdot
1_R\in R$ for each $s\in S$. Therefore, $R=S$.
\end{proof}

Let now $S$ be any ring and call $\Aut(S)$ the group of units of $S$. The canonical ring homomorphism $\nu:\mathbb{Z}\rightarrow S$ which takes $1_{\mathbb{Z}}$  to $1_S$ has kernel $0$ or $\mathbb{Z}_n$, for some $n\in\mathbb{N}$. In the first case, $S$ is called a ring of characteristic $0$ and, in the other, a ring of characteristic $n$. Let us denote $\mathbb Z$ by $\mathbb{Z}_0$. 

Throughout the rest of this section, let us denote by $S'$, the image of the group ring $\mathbb{Z}_n[\Aut(S)]$ inside $S$ under the ring homomorphism sending an element of $\Aut(S)$ to the corresponding element in $S$. 

Then $S'$ is a subring of $S$ consisting of those elements which can be written as a finite sum of units of $S$, where $n$ is the characteristic of the ring $S$. By construction, the subring $S'$ is invariant under left (or right) multiplication by units of $S$. The problem stated by Fuchs of characterizing endomorphism rings which are additively generated by automorphisms reduces then to characterizing when $S=S'$. From Theorem \ref{key2} and Proposition \ref{no-F2}, we deduce the following partial answers of this question.

\begin{cor}\label{reg1}
Let $S$ be a von Neumann regular and right self-injective ring of characteristic $n$. Then $S'$ is also a von Neumann regular ring.
\end{cor}

\begin{cor}\label{reg2}
Let $S$ be a von Neumann regular and right self-injective ring of characteristic $n$. If $S$ has no homomorphic images isomorphic to $\mathbb{Z}_2$,
then $S=S'$. In particular, this is the case when $n>0$ and $2$ does not divide $n$.
\end{cor}

\begin{proof}
The first part is an immediate consequence of Proposition \ref{no-F2}. For the second part, just note that if $\mathbb{Z}_n$ is the kernel of $\nu:\mathbb{Z}\rightarrow S$ for some $n>0$ and $2$ does not divide $n$, then $S$ cannot have any ring homomorphism $\delta:S\rightarrow \mathbb{Z}_2$ since otherwise, $2$ and $n$ would belong to $\Ker(\delta\circ\nu)$. And, as $2$ and $n$ are coprime, we would deduce that $1\in \Ker(\delta\circ\nu)$, contradicting that $\nu$ is a ring homomorphism.
\end{proof}

\bigskip

\section{Automorphism-invariant Modules}

\noindent Let us fix a non empty class of right $R$-modules
$\mathcal X$, closed under isomorphisms. We recall from \cite{E,Xu}
that an $\mathcal X$-{\em preenvelope} of a right module $M$ is a
morphism $u: M\rightarrow X$ with $X\in \mathcal{X}$ such that any
other morphism $g: M\rightarrow X^\prime$ with $X^\prime \in
\mathcal{X}$ factors through $u$. A preenvelope $u: M\rightarrow
X$ is called an $\mathcal{X}$-{\it envelope} if, moreover, it
satisfies that any endomorphism $h:X\rightarrow X$ such that
$h \circ u=u$ must be an automorphism. An
$\mathcal{X}$-(pre)envelope $u: M\rightarrow X$ is called {\it
monomorphic} if $u$ is a monomorphism. It is easy to check that this is the case when $\mathcal X$ contains the class of injective modules.

\begin{defi} \rm Let $M$ be a module and $\mathcal{X}$, a class of
$R$-modules closed under isomorphisms. We will say that $M$ is $\mathcal{X}$-{\it automorphism}-{\it
invariant} if there exists an $\mathcal{X}$-envelope $u: M\rightarrow
X$ satisfying that for any automorphism $g:
X\rightarrow X$ there exists an endomorphism $f:
M\rightarrow M$ such that $u\circ f=g\circ u$.
\end{defi}

\begin{remarks}\label{galois}\rm

(1) Let $M$ be an $\mathcal X$-automorphism-invariant module and $u:M\rightarrow X$, its monomorphic $\mathcal X$-envelope. Let us choose $g\in\Aut(X)$ and $f\in\End(M)$ such that $u\circ f=g\circ u$. As $g$ is an automorphism, there exists an $f'\in\End(M)$ such that $u\circ f'=g^{-1}\circ u$ and thus, $u\circ f\circ f'=u\circ f'\circ f=u$. Therefore, $f\circ f'$ and $f'\circ f$ are automorphisms by the definition of monomorphic envelope. This shows that $f\in\Aut(M)$.

\medskip

(2) The above definition is equivalent to assert that $M$ is invariant under the group action on $X$ given by $\Aut(X)$.
Moreover, in this case, (1) shows that the map
$$\Delta:\Aut(X)\rightarrow \Aut(M)$$
which assigns $g\mapsto f$
is a surjective group homomorphism whose kernel consists of those automorphisms $g\in\Aut(X)$ such that $g\circ u=u$.
This subgroup of $\Aut(X)$ is ussually called the {\em Galois group} of the envelope $u$ (see e.g. \cite{EEG}) and we will denote it by $\Gal(u)$. Therefore, we get that, for modules $M$ having monomorphic $\mathcal X$-envelopes, $M$ is $\mathcal X$-automorphism invariant precisely when the envelope $u$ induces a group isomorphism $\Aut(M)\cong \Aut(X)/Gal(u)$.

\medskip

(3) The above definition of $\mathcal X$-automorphism-invariant modules can be easily extended to modules having $\mathcal X$-preenvelopes. We have restricted our definition
to modules having  envelopes because these are the modules to which our results will be applied in practice.

\medskip

(4) If $\mathcal X$ is the class of injective modules, then $\mathcal{X}$-automorphism-invariant modules are precisely the automorphism-invariant modules studied in \cite{DF,ESS,GS1,GS2,LZ,SS2}.
\end{remarks}

The following definition is inspired by the notion of quasi-injective modules, as well as of strongly invariant module introduced in \cite[Definition, p. 430]{ZZ1}.

\begin{defi}\rm
Let $M$ be a module and $\mathcal{X}$, a class of modules closed
under isomorphisms. We will say that $M$ is $\mathcal
X$-endomorphism-invariant if there exists an $\mathcal X$-envelope
$u: M\rightarrow X$ satisfying that for any endomorphism
$g:X\rightarrow X$, there exists an endomorphism $f:M\rightarrow
M$ such that $g\circ u=u\circ f$.

\end{defi}

Note that if $\mathcal X$ is the class of injective modules, then
the $\mathcal{X}$-endomorphism-invariant modules are precisely the
quasi-injective modules \cite{JW}. Whereas if $\mathcal X$ is the class of pure-injective modules, the $\mathcal X$-endomorphism invariant modules are just the modules which are strongly invariant in their pure-injective envelopes in the sense of \cite[Definition, p. 430]{ZZ1}.

The following example shows that, in general, $\mathcal X$-automorphism-invariant modules need not be $\mathcal X$-endomorphism-invariant.

\begin{exam}\label{no-qi}
\rm Let $R$ be the ring of all eventually constant sequences
$(x_n)_{n \in \mathbb{N}}$ of elements in  $\mathbb{Z}_2$, and
$\mathcal X$, the class of injective right $R$-modules. As $R$ is
von Neumann regular, the class of injective $R$-modules coincides
with the class of pure-injective $R$-modules as well as with the
class of flat cotorsion $R$-modules (see \cite{Xu} for its
definition). The $\mathcal X$-envelope of $R_R$ is $u:
R_R\rightarrow X$ where $X=\prod_{n\in \mathbb{N}}\mathbb{Z}_2$.
Clearly, $X$ has only one automorphism, namely the identity
automorphism. Thus, $R_R$ is $\mathcal X$-automorphism-invariant
but it is not $\mathcal X$-endomorphism-invariant. Let us note
that this example also shows that the hypothesis that $S$ has no
homomorphic images isomorphic to $\mathbb{Z}_2$
cannot be removed from Proposition \ref{no-F2}.
\end{exam}

\begin{notation} \rm Along the rest of this section, $\mathcal{X}$ will denote a class of modules closed under isomorphisms and
$M$, a module with $u: M\rightarrow X$, a monomorphic
$\mathcal{X}$-envelope such that $\End(X)/J(\End(X))$ is a von Neumann regular right self-injective ring and idempotents lift modulo $J(\End(X))$.
\end{notation}

If $f:M\rightarrow M$ is an endomorphism, then we know by the definition of preenvelope that
$u\circ f$ extends to an endomorphism $g: X\rightarrow X$ such
that $g\circ u=u\circ f$. The following easy lemma asserts that this extension is unique modulo the Jacobson radical of $\End(X)$.

\begin{lem}\label{unique-J}
Let $f\in\End(M)$ be any endomorphism and assume that
$g,g'\in\End(X)$ satisfy that $g\circ u=u\circ f=g'\circ u$. Then
$g-g'\in J(\End(X))$.
\end{lem}

\begin{proof}
In order to prove that $g-g'\in J(\End(X))$, we must
show that $1-t\circ (g-g^\prime)$ is an automorphism for any $t\in
\End(X)$. Let us note that $t\circ (g-g^\prime)\circ u=t\circ
(g\circ u-g^\prime \circ u)=t\circ u\circ (f-f)=0$. Therefore, $[1-t\circ
(g-g^\prime)]\circ u=u$. Therefore, $1-t\circ (g-g^\prime)$
is an automorphism by the definition of envelope.
\end{proof}

The above lemma shows that
we can define a ring homomorphism

$$
\varphi: \End(M)\longrightarrow \End(X)/J(\End(X))
$$
by the
rule $\varphi(f)=g+J(\End(X))$.
Call $K=\Ker \varphi$. Then $\varphi$ induces an injective
ring homomorphism
$$
\Psi: \End(M)/K\longrightarrow \End(X)/J(\End(X))
$$
which allows us to identify $\End(M)/K$ with the subring $\Im
\Psi\subseteq \End(X)/J(\End(X)).$

\begin{lem}\label{converselemma} Assume $j\in J=
J(\End(X))$. Then there exists an element $k\in K$
such that $u\circ k=j\circ u$.
\end{lem}

\begin{proof} Let $j\in J$. Then $1-j$ is invertible. Since $M$ is
$\mathcal{X}$-automorphism-invariant, there exists a homomorphism
$f: M\rightarrow M$ such that $u\circ f=(1-j)\circ u$. Now, as
$j=1-(1-j)$, we have that $j\circ u=u-(1-j)\circ u=u\circ (1-f)$.
Therefore, $u\circ(1-f)=j\circ u$ and this means that
$\varphi(1-f)=j+J=0+J$ and $1-f\in \Ker \varphi=K$.
\end{proof}

We can now state one of the main result of this section.

\begin{theor} \label{main1}
If $M$ is $\mathcal{X}$-automorphism-invariant, then $\End(M)/K$ is von Neumann regular, $K=J(\End(M))$ and idempotents in $\End(M)/J(\End(M))$ lift to idempotents in $\End(M)$.
\end{theor}

\begin{proof} Call $S=\End(X)$, $J=J(\End(X))$ and $T=\Im \Psi\cong \End(M)/K$. We
want to show that we are in the situation of Theorem \ref{key2} to
get that $\End(M)/K$ is von Neumann regular. In order to prove it,
we only need to show that $T$ is invariant under left
multiplication by units of $S/J$. Let $g+J$ be a  unit of $S/J$.
Then $g: X\rightarrow X$ is an automorphism. And this means that
there exists an $f: M\rightarrow M$ such that $u\circ f=g\circ u$.
Therefore, $\Psi(f+K)=g+J$. Finally, if $\Psi(f^\prime+K)$ is
any element of $T$, we have that
$$
(g+J)\Psi(f^\prime
+K)=\Psi(f+K)\Psi(f^\prime+K)=\Psi(ff^\prime+K)\in \Im \Psi.
$$
This shows that $T$ is invariant under left multiplication by
units of $S/J$, namely $\End(M)/K$ is von Neumann regular.

As $\End(M)/K$ is von Neumann regular, $J(\End(M)/K)=0$ and so,
$J(\End(M))\subseteq K$. Let us prove the converse. As $K$ is a
two-sided ideal of $\End(M)$ it is enough to show that  $1-f$ is
invertible for every $f\in K$. Let $f\in K$ and let $g:
X\rightarrow X$ such that $g\circ u=u\circ f$. As $f\in K$, we get
that $g+J=\Psi(f+K)=0$. Therefore, $1-g$ is a unit in $S$. As
$(1-g)^{-1}: X\rightarrow X$ is an automorphism, there exists an
$h: M\rightarrow M$ such that $(1-g)^{-1}\circ u=u\circ h$. But
then
$$
u=(1-g)^{-1}\circ (1-g)\circ u=(1-g)^{-1}\circ u\circ (1-f)=u\circ
h\circ (1-f)
$$
and
$$
u=(1-g)\circ (1-g)^{-1}\circ u=(1-g)\circ u\circ h=u\circ
(1-f)\circ h.
$$
And, as $u:M\rightarrow X$ is a monomorphic envelope, this implies that both
$(1-f)\circ h$ and $h\circ (1-f)$ are automorphisms. Therefore,
$1-f$ is invertible.

Finally, let us prove that idempotents lift modulo $J(\End(M))$.
Let us choose an $f\in \End(M)$ such that $f+K=f^2+K$. Then there
exists a homomorphism $g: X\rightarrow X$ such that $g\circ
u=u\circ f$. Therefore, $g+J=\Psi(f+K)$. And, as $f+K$ is
idempotent in $\End(M)/K$, so is $g+J\in S/J$. As idempotents lift
modulo $J$, there exists an $e=e^2\in S$ such that $g+J=e+J$. Now
$g-e\in J$ implies that there exists an $k\in K$ such that
$(g-e)\circ u=u\circ k$ by Lemma \ref{converselemma}. Note that
$u\circ (f-k)=e\circ u$ and thus, $\varphi(f-k)=e+J$.

Therefore,
$$
u\circ (f-k)^2=e\circ u\circ (f-k)=e^2\circ u=e\circ u=u\circ
(f-k).
$$
And, as $u$ is monic, we get that $(f-k)^2=f-k$. This shows that idempotents
lift modulo $J(\End(M))$.
\end{proof}

\medskip
Recall that the notion of exchange property for modules was introduced by  Crawley and J\'{o}nnson in \cite{CJ}. A right $R$-module $M$ is said to satisfy the {\em exchange property} if for every right $R$-module $A$ and any two direct sum decompositions $A=M^{\prime}\oplus N=\oplus_{i\in \mathcal I}A_{i}$ with $M^{\prime} \simeq M$, there exist submodules $B_i$ of $A_i$ such that $A=M^{\prime} \oplus (\oplus_{i \in \mathcal I}B_i)$.
If this hold only for $|\mathcal I|<\infty$, then $M$ is said to satisfy the {\em finite exchange property}. Crawley and J\'{o}nnson raised the question whether the finite exchange property always implies the full exchange property but this question is still open. A ring $R$ is called an {\it exchange ring} if the module $R_R$ (or $_RR$) satisfies the (finite) exchange property. Warfield proved that a module $M$ has the finite exchange property if and only if $\End(M)$ is an exchange ring.

As a consequence of the above theorem, we have

\begin{cor} \label{finiteexchange}
If $M$ is $\mathcal{X}$-automorphism-invariant, then $M$ satisfies the finite exchange property.
\end{cor}

\begin{proof}
We have shown in Theorem \ref{main1} that $\End(M)/J(\End(M))$ is
a von Neumann regular ring and idempotents lift modulo
$J(\End(M))$. Then $\End(M)$ is an exchange ring by
\cite[Proposition 1.6]{Nichol}. This proves that $M$ has the
finite exchange property.
\end{proof}

Before proving our structure theorem for $\mathcal X$-automorphism-invariant modules, we need to state the following technical lemmas.

\begin{lem}\label{summands}
Let $M$ be an $\mathcal{X}$-automorphism-invariant module and assume that every direct summand of $M$ has an $\mathcal{X}$-envelope. If $N$ is a direct summand of $M$, then $N$ is also $\mathcal{X}$-automorphism-invariant.
\end{lem}

\begin{proof}
Let us write $M=N\oplus L$ and let $u_N:N\rightarrow X_N$ and $u_L:L\rightarrow X_L$ be the $\mathcal X$-envelopes of $N, L$, respectively. Then the induced morphism $u:M=N\oplus L\rightarrow X_N\oplus X_L$ is also an $\mathcal X$-envelope by \cite[Theorem 1.2.5]{Xu}. Let now $f:X_N\rightarrow X_N$ be any automorphism and consider the diagonal automorphism 
$(f,1_{X_L}):X_N\oplus X_L\rightarrow X_N\oplus X_L$. As $M$ is $\mathcal X$-automorphism-invariant, we get that 
$(f,1_{X_L})(M)\subseteq M$. But this means that $f(N)\subseteq N$ by construction.
\end{proof}

\begin{lem}\label{sum-units}
Let $M$ be an $\mathcal X$-automorphism-invariant module and $u:M\rightarrow X$, its $\mathcal X$-envelope. If any element of $\End(X)/J(\End(X))$ is the sum of two units, then $M$ is $\mathcal X$-endomorphism-invariant.
\end{lem}

\begin{proof}
We claim that any element in $\End(X)$ is also the sum of two units. Let $s\in \End(X)$. Then $s+J(\End(X))$ is the sum of two units, say $s+J(\End(X))= (s'+ s'')+J(\End(X))$. This yields that $s', s''$
are units in $\End(X)$ and there exists a $j\in J(\End(X))$ such
that $s=s'+(s''+j)$. Now, note that $s''+j$ is also a unit.

On the other hand, as any element in $\End(X)$ is the sum of two units, and $M$ is $\mathcal X$-automorphism-invariant, we get that
$M$ is $\mathcal X$-endomorphism-invariant.
\end{proof}

Let $M=N\oplus L$ be a decomposition of a module $M$ into two direct summands and call $v_N:N\rightarrow M, v_L:L\rightarrow M,\pi_N:M\rightarrow N,\pi_L:M\rightarrow L$ the associated structural injections and projections. We may associate to any homomorphism $f\in {\rm Hom}(N,L)$, the endomorphism $v_L\circ f\circ \pi_N$ of $M$, and thus identify ${\rm Hom}(N,L)$ with a subset of $\End(M)$. Similarly, we identify ${\rm Hom}(L,N)$ with a subset of $\End(M)$ as well. We will use these identifications in the following theorem:

\begin{theor} \label{struct-invar}
If $M$ is $\mathcal{X}$-automorphism-invariant and every direct summand of $M$ has an $\mathcal X$-envelope, then $M$ admits a decomposition $M=N\oplus L$ such that:
\begin{enumerate}
\item[(i)] $N$ is a square-free module.
\item[(ii)] $L$ is $\mathcal X$-endomorphism-invariant and $\End(L)/J(\End(L))$ is
von Neumann regular, right self-injective and idempotents lift
modulo $J(\End(L))$.
\item[(iii)] Both $\H_R(N,L)$ and $\H_R(L,N)$ are contained in $J(\End(M))$.
\end{enumerate}
In particular, $\End(M)/J(\End(M))$ is the direct product of an abelian regular ring and a right self-injective
von Neumann regular ring.
\end{theor}

\begin{proof}
Let us call $S=\End(X)$ and decompose, as in Lemma \ref{type},
$S/J(S)=T_1\times T_2$, where $T_1$ is an abelian regular
self-injective ring and $T_2$ is a right self-injective ring in
which every element is the sum of two units. Let
$\Psi:\End(M)/J(\End(M))\rightarrow S/J(S)$ be the injective ring
homomorphism defined in the beginning of this section. Identifying
$\End(M)/J(\End(M))$ with $\Im(\Psi)$, we get that
$\End(M)/J(\End(M))$ is a subring of $S/J(S)$ invariant under left
multiplication by units of $S/J(S)$. Using now Theorem
\ref{key2}, we obtain that $\End(M)/J(\End(M))=R_1\times
R_2$, where $R_1$ is an abelian regular ring and
$R_2$ is a right self-injective von Neumann regular ring which is invariant under
left multiplication by elements in $S/J(S)$. Let
$e+J(\End(M))$ be a central idempotent of $\End(M)/J(\End(M))$
such that $R_1=e\cdot \End(M)/J(\End(M))$ and $R_2=(1-e)\cdot
\End(M)/J(\End(M))$. As idempotents lift modulo $J(\End(M))$, we
may choose $e$ to be an idempotent of $\End(M)$. Call then $N=eM$
and $L=(1-e)M$. Note first that both $\H_R(N,L)$ and $\H_R(L,N)$
are contained in 
$J(\End(M))$ since 
$$\End(M)/J(\End(M))=[e\cdot \End(M)/J(\End(M))]\times [(1-e)\cdot
\End(M)/J(\End(M))].$$ This shows (iii).
Moreover, as $e+J(\End(M))$ is central in $\End(M)/J(\End(M))$, we
get that 
$$\End(N)/J(\End(N))\cong R_1$$ 
and 
$$\End(L)/J(\End(L))\cong R_2.$$ 

\noindent Now, as both $N, L$ are direct summands of $M$, they are $\mathcal X$-automorphism invariant by Lemma \ref{summands}. In particular, we get that idempotents in $\End(N)/J(\End(N))$ and in $\End(L)/J(\End(L))$ lift to idempotents in $\End(N)$ and $\End(L)$, respectively. Next, we claim that $N$ is square-free. Assume to the contrary that $N=N_1\oplus N_2\oplus N_3$ with $N_1\cong N_2\neq 0$ and let $e_1,e_2,e_3\in \End(N)$ be   orthogonal idempotents such that $N_i=e_iN$ for each $i=1,2,3$. Then $e_1R_1\cong e_2R_1$. Let $\phi:e_1R_1\rightarrow e_2R_1$ be an isomorphism and call $e_2r=\phi(e_1)$. As each idempotent in $R_1$ is central, we get that $\phi(e_1)=\phi(e^2_1)=e_2re_1=re_1e_2=0$. This yields a contradiction, since $e_1,e_2$ are nonzero idempotents in $\End(N)$. 
Thus we have proved (i). Finally, as $R_2$ is invariant under left multiplication by elements in $S/J(S)$, it follows that $L$ is $\mathcal X$-endomorphism-invariant. This proves (ii), thus completing the proof of the theorem. 
\end{proof}

Our next theorem extends \cite[Theorem 11]{ZZ1} and gives new examples of modules satisfying the full exchange property. Let us note that, for instance, Example \ref{no-qi} is not covered by \cite[Theorem 11]{ZZ1}.

\begin{theor} \label{fullexchange}
Let $M$ be an $\mathcal{X}$-automorphism-invariant module and assume every direct summand of $M$ has an $\mathcal X$-envelope. Assume further that for $\mathcal X$-endomorphism-invariant modules, the finite exchange property implies the full exchange property. Then $M$ satisfies the full exchange property.
\end{theor}

\begin{proof}
By the theorem above, we have the decomposition $M=N\oplus L$
where $N$ is a square-free module and $L$ is an $\mathcal
X$-endomorphism-invariant module. Now, in Corollary
\ref{finiteexchange}, we have seen that $M$ satisfies the finite
exchange property. Therefore, both $N$ and $L$ satisfy the finite
exchange property. By our hypothesis, $L$ satisfies the full exchange
property. It is known that for a square-free module, the finite
exchange property implies the full exchange property \cite[Theorem
9] {Nielsen}. Since a direct sum of two modules with the full
exchange property also has the full exchange property, it follows
that $M$ satisfies the full exchange property.
\end{proof}

Recall that a ring $R$ is called a {\it clean ring} if each
element $a\in R$ can be expressed as $a=e+u$ where $e$ is an
idempotent in $R$ and $u$ is a unit in $R$ \cite{Nichol}. A module $M$ is called
a {\it clean module} if $\End(M)$ is a clean ring.

\begin{theor} \label{clean}
Let $M$ be an $\mathcal{X}$-automorphism-invariant module and assume that every direct summand of $M$ has an $\mathcal X$-envelope. Then $M$ is a clean module.
\end{theor}

\begin{proof}
By Theorem \ref{struct-invar}, $\End(M)/J(\End(M))$ is the direct product of an abelian regular ring and a right self-injective
von Neumann regular ring. We know that abelian regular rings and right self-injective rings are clean. Since direct product of clean rings is clean, it follows that $\End(M)/J(\End(M))$ is clean. We have also shown in
Theorem \ref{main1} that idempotents in $\End(M)/J(\End(M))$ lift to idempotents in $\End(M)$. Therefore, $\End(M)$ is a clean ring by Nicholson \cite[Page 272]{Nichol}. Thus $M$ is a clean module.
\end{proof}

In particular, when $\End(M)$ has no homomorphic images isomorphic
to ${\mathbb Z}_2$, we have:

\begin{theor}\label{F2-inv}
Let $M$ be $\mathcal{X}$-automorphism-invariant. If $\End(M)$ has
no homomorphic images isomorphic to $\mathbb{Z}_2$, then $M$ is $\mathcal X$-endomorphism-invariant and
therefore, $$\End(M)/J(\End(M))= \End(X)/J(\End(X)).$$ In
particular, $\End(M)/J(\End(M))$ is von Neumann regular, right
self-injective and idempotents lift modulo $J(\End(M))$.

This is the case when $char(\End(M))=n>0$ and $2\nmid n$.
\end{theor}

\begin{proof}
Assume that $\End(M)$ has no homomorphic images isomorphic to
$\mathbb{Z}_2$. Then neither has
$\End(M)/J(\End(M))$ and thus, we deduce that
$$\End(M)/J(\End(M))=\End(X)/J(\End(X))$$ by Proposition
\ref{no-F2}. The proof of Proposition \ref{no-F2} shows that
$\End(X)/J(\End(X))$ has no homomorphic images isomorphic to
$\mathbb{Z}_2$ and thus, any element in
$\End(X)/J(\End(X))$ is the sum of two units. Applying now Lemma \ref{sum-units}, we deduce that 
$M$ is $\mathcal X$-endomorphism-invariant.

Finally, if $char(\End(M))=n>0$ and $2\nmid n$, then $\End(M)$ cannot have homomorphic images isomorphic to $\mathbb{Z}_2$ for the same reason as in Corollary \ref{reg2}.
\end{proof}

\bigskip

\section{Automorphism-coinvariant modules}

\noindent We will devote this section to dualize the results
obtained in Section 3. Let $M$ be a module and $\mathcal{X}$, a
class of $R$-modules closed under isomorphisms. A homomorphism $p:
X\rightarrow M$ is an $\mathcal{X}$-{\it precover} if any other
$g:X^\prime\rightarrow M$ with $X^\prime\in \mathcal{X}$
factorizes through it. And an $\mathcal{X}$-precover is called an
$\mathcal{X}$-{\it cover} if, moreover, any $h: X\longrightarrow
X$ such that $p \circ h=p$ must be an automorphism \cite{Xu}. An
$\mathcal{X}$-cover $p: X\longrightarrow M$ is called {\it
epimorphic} if $p$ is an epimorphism.

\begin{defi} \rm Let $M$ be a module and $\mathcal{X}$, a class of
$R$-modules closed under isomorphisms. We will say that $M$ is $\mathcal{X}$-{\it automorphism}-{\it
coinvariant} if there exists an $\mathcal{X}$-cover $p: X\rightarrow
M$ satisfying  that
 for any automorphism $g: X\rightarrow X$,
there exists an endomorphism $f: M\rightarrow M$ such that
$f \circ p=p \circ g$.
\end{defi}

\begin{remarks}\label{dualgalois}\rm

(1) As in Remarks \ref{galois}, the above definition can be easily extended to modules having $\mathcal X$-precovers. Moreover, if $p:X\rightarrow M$ is an epimorphic cover, then $M$ is $\mathcal X$-automorphism invariant precisely when the cover $p$ induces a group isomorphism $\Delta':\Aut(M)\cong \Aut(X)/\coGal(X)$, where $\coGal(X)=\{g\in\Aut(X)\,|\,p\circ g=p\}$ is usually called the {\em co-Galois} group of the cover $p$ (see e.g. \cite{EGJO,EEG}).

\medskip
(2) If $\mathcal X$ is the class of projective modules, then $\mathcal{X}$-automorphism-coinvariant modules are precisely dual
automorphism-invariant modules which have projective covers studied in \cite{SS1}.
\end{remarks}

The following definition is inspired by the notion of quasi-projective modules.

\begin{defi}\rm
Let $M$ be a module and $\mathcal{X}$, a class of modules closed
under isomorphisms. We will say that $M$ is $\mathcal
X$-endomorphism-coinvariant if there exists an $\mathcal X$-cover
$p: X\rightarrow M$ satisfying that for any endomorphism
$g:X\rightarrow X$, there exists an endomorphism $f:M\rightarrow
M$ such that $p\circ g=f\circ p$.
\end{defi}

Note that if $\mathcal X$ is the
class of projective modules, then the
$\mathcal{X}$-endomorphism-coinvariant modules are precisely the
quasi-projective modules \cite{WJ}.

\begin{notation} \rm Throughout this section $\mathcal{X}$ will be a class of modules closed under isomorphisms,
$M$ a module with $p: X\rightarrow M$ an epimorphic
$\mathcal{X}$-cover such that $\End(X)/J(\End(X))$ is a von Neumann regular, right self-injective ring and idempotents lift modulo $J(\End(X))$.
\end{notation}

 If $f:M\rightarrow M$ is an endomorphism, then there exists a $g:
X\rightarrow X$ such that $p\circ g=f\circ p$. Moreover, if
$g^\prime :X\rightarrow X$ also satisfies that $p\circ g^\prime =
f\circ p$, then we get that $p\circ (g-g^\prime)=0$. Thus, $p\circ
(g-g^\prime)\circ t=0$ for any $t\in S=\End(X)$ and this means
that $p\circ (1-(g-g^\prime)\circ t)=p$. Then by the definition of
cover $1-(g-g^\prime)\circ t$ is an automorphism. Thus,
$1-(g-g^\prime)\circ t$ is an isomorphism for all $t\in S$ and we
get that $g-g^\prime \in J(S)$. Therefore we can define a ring
homomorphism
$$
\varphi: \End(M)\rightarrow S/J(S)\, \mbox {by} \,
\varphi(f)=g+J(S).
$$
Call $K=\Ker(\varphi)$ and $J=J(S)$. Then $\varphi$ induces an injective ring
homomorphism $\Psi: \End(M)/K\rightarrow S/J$. A dual argument to the one used in Lemma \ref{converselemma} proves that:

\begin{lem}\label{dual1} Assume that
$j\in J$. Then there exists an element $k\in K$ such
that $p\circ j=k \circ p$.
\end{lem}

\begin{theor}\label{dual3} If $M$ is $\mathcal{X}$-automorphism-coinvariant, then $\End(M)/K$ is von Neumann regular, $K=J(\End(M))$ and
idempotents lift modulo $J(\End(M))$.
\end{theor}

\begin{proof} Using Theorem \ref{key2}, in order to show that $\End(M)/K$
is von Neumann regular we only need to show that $\Im
\Psi$ is invariant under left multiplication by units of $S/J$. This can be proved in a similar way as in Theorem \ref{main1}.

As $\End(M)/K$ is von Neumann regular, clearly
$J(\End(M))\subseteq K$. Let us prove the converse. As $K$ is a
two-sided ideal, we only need to show that $1-f$ is invertible in
$\End(M)$ for every $f\in K$. So take $f\in K$. Then
$\Psi(1-f+K)=1+J$. Let $g: X\rightarrow X$ such that $(1-f)\circ
p=p\circ g$. Then $1-g\in J$ and thus, $g=1-(1-g)$ is invertible.
So there exists an $h: M\rightarrow M$ such $p\circ g^{-1}=h\circ p$. Therefore,
$$h\circ (1-f)\circ p=h\circ p\circ g=p\circ g^{-1}\circ g=p,
(1-f)\circ h\circ p=(1-f)\circ p\circ g^{-1}=p\circ g\circ
g^{-1}=p.
$$
And, as $p: X\rightarrow M$ is an epimorphic cover, we get that
$(1-f)\circ h$ and $h\circ (1-f)$ are automorphisms. Therefore
$1-f$ is invertible. So, $K=J(\End(M))$.

Finally, the proof that idempotents lift modulo $J(End(M))$ is also dual to the proof in Theorem \ref{main1} changing envelopes by covers.
\end{proof}

The proofs similar to Corollary \ref{finiteexchange}, Lemma \ref{summands}, Theorems
\ref{struct-invar}, \ref{fullexchange}, \ref{clean} and
\ref{F2-inv} also show:

\begin{cor} \label{dual-finiteexchange}
If $M$ is $\mathcal{X}$-automorphism-coinvariant, then $M$ satisfies the finite exchange property.
\end{cor}

\begin{lem} \label{dual-summands}
If $M$ is $\mathcal{X}$-automorphism-coinvariant and every direct summand of $M$ has an $\mathcal X$-cover, then any direct summand of $M$ is also $\mathcal{X}$-automorphism-coinvariant.
\end{lem} 

As discussed earlier in the paragraph above Theorem \ref{struct-invar}, for $M=N\oplus L$, we are again identifying ${\rm Hom}(N,L)$ and ${\rm Hom}(L,N)$ with appropriate subsets of $\End(M)$ in the next theorem.

\begin{theor} \label{struct-coinvar}
If $M$ is $\mathcal{X}$-automorphism-coinvariant and every direct summand of $M$ has an $\mathcal X$-cover, then $M$ admits a decomposition $M=N\oplus L$ such that:
\begin{enumerate}
\item[(i)] $N$ is a square-free module.
\item[(ii)] $L$ is $\mathcal X$-endomorphism-coinvariant and
$\End(L)/J(\End(L))$ is von Neumann regular, right self-injective
and idempotents lift modulo $J(\End(L))$.
\item[(iii)] Both $\H_R(N,L)$ and $\H_R(L,N)$ are contained in
$J(\End(M))$.
\end{enumerate}
In particular, $\End(M)/J(\End(M))$ is the direct product of an abelian regular ring and a right self-injective
von Neumann regular ring.
\end{theor}

\begin{theor} \label{dual-fullexchange}
Let $M$ be $\mathcal{X}$-automorphism-coinvariant and assume every direct summand of $M$ has an $\mathcal X$-cover. Assume further that for $\mathcal X$-endomorphism-coinvariant
modules, the finite exchange property implies the full exchange
property. Then $M$ satisfies the full exchange property.
\end{theor}

\begin{theor} \label{dual-clean}
Let $M$ be $\mathcal{X}$-automorphism-coinvariant and assume that every direct summand of $M$ has an $\mathcal X$-cover. Then $M$ is a clean module.
\end{theor}

\begin{theor}\label{dual-F2-inv}
Let $M$ be $\mathcal{X}$-automorphism-coinvariant. If $\End(M)$
has no homomorphic images isomorphic to $\mathbb{Z}_2$, then $M$ is $\mathcal X$-endomorphism-coinvariant
and $\End(M)/J(\End(M))=\End(X)/J(\End(X))$. In particular,
$\End(M)/J(\End(M))$ is von Neumann regular, right self-injective
and idempotents lift modulo $J(\End(M))$.

This is the case when $char(\End(M))=n>2$ and $2\nmid n$.
\end{theor}

\bigskip

\section{Applications}

\noindent We are going to finish the paper by showing how our results can be applied to  a wide variety of classes of modules obtaining interesting consequences for them.

\bigskip

\noindent {\bf Application 1.} Let $\mathcal{X}$ be the class of injective modules and $M$, an $\mathcal{X}$-automorphism-invariant module. Let $E$ be the injective envelope of $M$. Then $S=\End(E)$ satisfies that $S/J(S)$ is von Neumann regular, right
self-injective and idempotents lift modulo $J(S)$. Therefore,

\medskip \noindent (a) \cite[Proposition 1]{GS2}. By Theorem \ref{main1}, it follows that
$\End(M)/J(\End(M))$ is von Neumann regular and idempotents lift
modulo $J(\End(M))$. This extends the corresponding result for quasi-injective modules by Faith and Utumi \cite{FU}.

\medskip \noindent (b) \cite[Theorem 3]{ESS}. By Theorem \ref{struct-invar}, $M=N\oplus L$ where $N$ is square-free and $L$ is quasi-injective.

 \medskip \noindent (c) \cite[Theorem 3]{GS2}. Since every quasi-injective module satisfies the full exchange property \cite{Fuchs}, we deduce from
 Theorem \ref{fullexchange} that $M$ also satisfies the full exchange property. This extends results of Warfield \cite{Warfield1} and Fuchs \cite{Fuchs}.

 \medskip \noindent (d) \cite[Corollary 4]{GS2}. By Theorem \ref{clean}, $M$ is a clean module.

 \medskip \noindent (e) \cite[Theorem 3]{GS1}. By Theorem \ref{F2-inv}, if $\End(M)$ has no homomorphic images isomorphic to $\mathbb{Z}_2$, then $M$ is quasi-injective.

\bigskip

\noindent {\bf Application 2.} Let $\mathcal{X}$ be the class of pure-injective modules and
$M$, an $\mathcal{X}$-automorphism-invariant module. Let $E$ be the
pure-injective envelope of $M$. Then $S=\End(E)$ satisfies that
$S/J(S)$ is von Neumann regular, right self-injective and
idempotents lift modulo $J(S)$ (see e.g. \cite{GH}). Therefore,

\medskip \noindent (a) By Theorem \ref{main1}, it follows that $\End(M)/J(\End(M))$ is
von Neumann regular and idempotents lift modulo $J(\End(M))$. This extends the characterization of the endomorphism rings of pure-injective modules obtained by Huisgen-Zimmermann and Zimmermann \cite{ZZ}.

\medskip \noindent (b) By Theorem \ref{struct-invar}, $M=N\oplus L$ where $N$ is square-free and $L$ is invariant under endomorphisms of its pure-injective envelope. In fact, $L$ turns out to be strongly invariant in its pure-injective envelope in the sense of \cite[Definition, p. 430]{ZZ1}.

\medskip \noindent (c) It follows from \cite[Theorem 11]{ZZ1} that a module invariant under endomorphisms of its pure-injective envelope satisfies the full exchange property. Therefore, we deduce from Theorem \ref{fullexchange} that $M$ also satisfies the full exchange property, thus extending \cite[Examples, p. 431]{ZZ1}.

 \medskip \noindent (d) By Theorem \ref{clean}, $M$ is a clean module.

\bigskip

\noindent {\bf Application 3.} Let $\mathcal{X}$ be the class of projective modules over a right perfect ring and $M\in \M$-$R$, an $\mathcal{X}$-automorphism-coinvariant module, then

\medskip \noindent (a) By Theorem \ref{dual3},
$\End(M)/J(\End(M))$ is von Neumann regular and idempotents lift
modulo $J(\End(M))$. 

\medskip \noindent (b) By Theorem \ref{struct-coinvar}, $M=N\oplus L$ where $N$ is square-free and $L$ is quasi-projective.

\medskip \noindent (c) By (a), $M$ satisfies the finite exchange property. Moreover, we know by (b) that $M=N\oplus L$
where $N$ is square-free and $L$ is quasi-projective. Thus $N$
satisfies the finite exchange property. Since $N$ is square-free,
this implies that $N$ satisfies the full exchange property. It is known
that a quasi-projective right modules over a right perfect ring is
discrete (see \cite[Theorem 4.41]{mm}). Thus $L$ is a discrete
module. Since discrete modules satisfy the exchange property, we
have that $L$ satisfies the full exchange property.

Thus it follows that $M$ satisfies the full exchange property.

\medskip \noindent (d) By Theorem \ref{dual-clean}, $M$ is clean.

\medskip \noindent (e) By Theorem \ref{dual-F2-inv}, if $\End(M)$ has no homomorphic image isomorphic to $\mathbb {Z}_2$, then $M$ is quasi-projective.

\bigskip
\noindent {\bf Application 4.} Let $\mathcal{X}$ be the class of projective modules over a semiperfect ring and $M\in \M$-$R$, a finitely generated $\mathcal{X}$-automorphism-coinvariant module, then

\medskip \noindent (a) By Theorem \ref{dual3},
$\End(M)/J(\End(M))$ is von Neumann regular and idempotents lift
modulo $J(\End(M))$. 

\medskip \noindent (b) By Theorem \ref{struct-coinvar}, $M=N\oplus L$ where $N$ is square-free and $L$ is quasi-projective.

\medskip \noindent (c) By (a), $M$ satisfies the finite exchange property. Moreover, we know by (b) that $M=N\oplus L$
where $N$ is square-free and $L$ is quasi-projective. Since $N$ is square-free,
this implies that $N$ satisfies the full exchange property. It is known
that a finitely generated quasi-projective right modules over a right semiperfect ring is
discrete (see \cite[Theorem 4.41]{mm}), therefore $L$ satisfies the full exchange property.

Thus it follows that $M$ satisfies the full exchange property.

\medskip \noindent (d) By Theorem \ref{dual-clean}, $M$ is clean.

\medskip \noindent (e) By Theorem \ref{dual-F2-inv}, if $\End(M)$ has no homomorphic image isomorphic to $\mathbb {Z}_2$, then $M$ is quasi-projective.

\bigskip
\noindent {\bf Application 5.} Let $(\mathcal{F}, \mathcal{C})$ be
a cotorsion pair (see e.g. \cite{GT}), i.e.,
$\mathcal{F}, \mathcal{C}$ are two classes of modules such that \\
(i) $F\in \mathcal{F}\Leftrightarrow$ Ext$^1(F, C)=0$ for all $C\in
\mathcal{C}$.\\
(ii) $C\in \mathcal{C}\Leftrightarrow$ Ext$^1(F, C)=0$ for all $F\in
\mathcal{F}$.

Let us assume that $\mathcal{F}$ is closed under direct limits and
that any module has an $\mathcal{F}$-cover (and therefore a
$\mathcal{C}$-envelope). This is true, for instance,  if there exists
a subset $\mathcal{F}_0\subset \mathcal{F}$ such that $C\in
\mathcal{C}\Leftrightarrow$ Ext$^1(F, C)=0$ for all $F\in
\mathcal{F}_0$ (see e.g. \cite[Theorem 2.6]{EEGO} or \cite[Chapter 6]{GT}). It is well-known
that in a cotorsion pair, any $\mathcal{C}$-envelope $u:
M\rightarrow C(M)$ is monomorphic and any $\mathcal{F}$-cover $p:
F(M)\rightarrow M$ is epimorphic.

In the particular case in which $\mathcal F$ is the class of flat modules, the cotorsion pair $(\mathcal{F},\mathcal{C})$ is usually called the {\em Enochs cotorsion pair} \cite[Definition 5.18, p. 122]{GT} and modules in $\mathcal C$ are just called {\em cotorsion modules}. By the definition of a cotorsion pair, we get that a module $M$ is cotorsion if and only if ${\rm Ext}^1(F, M)=0$ for every flat module. A ring $R$ is called {\em right
cotorsion} if $R_R$ is a cotorsion module.

\medskip
The following easy lemma is implicitly used in \cite{GH} without any proof. We are including a proof for the sake of completeness.

\begin{lem}\label{app1} Assume that $(\mathcal{F}, \mathcal{C})$
is a cotorsion pair and $\mathcal F$ is closed under direct limits. Let $X\in \mathcal{F}\cap \mathcal{C}$. Then
$S=\End(X)$ is a right cotorsion ring.

In particular, $S/J(S)$ is
von Neumann regular, right self-injective and idempotents lift modulo $J(S)$.
\end{lem}

\begin{proof} Let us show that $S$ is a right cotorsion ring. Let $F$ be a flat right $S$-module and let
$$\Sigma:\,\,0\rightarrow S\overset{u}{\longrightarrow} N\overset{p}{\longrightarrow}F\rightarrow 0$$
be any extension of $F$ by $S$. We must show that $\Sigma$ splits.
As $F$ is flat, we may write it as a direct limit of free modules
of finite rank, say $F=\underrightarrow{\text{lim}}\,S^{n_i}$.
Applying now the tensor functor $T=-\otimes_SX$,  we get in
Mod-$R$ the sequence
$$0\rightarrow T(S)\overset{T(u)}{\longrightarrow} T(N)\overset{T(p)}{\longrightarrow}T(F)\rightarrow 0$$
which is exact since $F$ is flat. Moreover, $T(S)\cong X$.

On the other hand, as $T=-\otimes_SX$ commutes with direct limits,
we get that
$$T(F)= T(\underrightarrow{\text{lim}}\,S^{n_i})\cong \underrightarrow{\text{lim}}\,T(S^{n_i})\cong \underrightarrow{\rm lim}\,X^{n_i}$$
and thus, $T(F)$ belongs to $\mathcal F$ since we are assuming that $\mathcal F$ is closed under direct limits. As $X_R\in\mathcal C$, we get that the above sequence splits and therefore, there exists a $\pi:T(N)\rightarrow T(S)$ such that $\pi\circ T(u)=1_{T(S)}$. Applying now the functor $H={\rm Hom}_R(X,-)$, we get a commutative diagram in $\M$-$S$
$$\begin{array}{ccccccccc}
  0&\rightarrow & S&\overset{u}{\longrightarrow}& N&\overset{p}{\longrightarrow}& F&\rightarrow& 0\\
  &&\downarrow^{\sigma_S}&&\downarrow^{\sigma_N}&&\downarrow^{\sigma_F}\\
  0&\rightarrow & HT(S) & \xrightarrow{HT(u)} &HT(N) & \xrightarrow{HT(p)}&HT(F) & &
\end{array}$$
in which $\sigma:1_{{\rm Mod}-S}\rightarrow HT$ is the arrow of the adjunction and $\sigma_S$ is an isomorphism. Then we have that
$$\sigma_S^{-1}\circ H(\pi)\circ\sigma_N\circ u=\sigma_S^{-1}\circ H(\pi)\circ HT(u)\circ \sigma_S)=\sigma_S^{-1}\circ 1_{HT(S)}\circ\sigma_S=1_S$$
and this shows that $\Sigma$ splits.

Therefore, $S=\End(X)$ is right cotorsion. Now, $S/J(S)$ is
von Neumann regular, right self-injective and idempotents lift modulo $J(S)$ by the main result of \cite{GH}.
\end{proof}

\medskip
\noindent It is well-known that if $(\mathcal{F},\mathcal{C})$ is a cotorsion pair and $u: M\rightarrow
C(M)$ is a $\mathcal{C}$-envelope, then Coker$(u)\in \mathcal{F}$ (see e.g. \cite{Xu}).
In particular, if $M\in \mathcal F$, then $C(M)\in \mathcal F$ (since  $\mathcal{F}$ is closed under extensions) and thus
$C(M)\in \mathcal{F}\cap \mathcal{C}$. 
Dually, if $p:
F(M)\rightarrow M$ is an $\mathcal{F}$-cover, then $\Ker p\in
\mathcal{C}$ and, if $M\in \mathcal C$, this means that $F(M)\in \mathcal C$, as  $\mathcal{C}$ is also closed under
extensions. Therefore, we also get that $F(M)\in \mathcal{F}\cap \mathcal{C}$.
As a consequence, we can apply our previous results to this
situation.

\begin{theor}\label{app2} Let $(\mathcal{F}, \mathcal{C})$ be a
cotorsion pair such that $\mathcal{F}$ is closed under direct limits and every module has an $\mathcal{F}$-cover.

\medskip

\begin{enumerate}
\item If $M\in \mathcal{F}$ is
$\mathcal{C}$-automorphism-invariant, then

\medskip
\noindent (i) $\End(M)/J(\End(M))$ is
von Neumann regular and idempotents lift modulo $J(\End(M))$. Consequently, $M$ satisfies the finite exchange property.

\medskip
\noindent (ii) $M=N\oplus L$ where $N$ is square-free and and $L$ is $\mathcal{C}$-endomorphism-invariant.

\medskip
\noindent (iii) $M$ is a clean module.

\medskip

\item If $M\in \mathcal{C}$ is
$\mathcal{F}$-automorphism-coinvariant, then

\medskip
\noindent (i) $\End(M)/J(\End(M))$ is
von Neumann regular and idempotents lift modulo $J(\End(M))$. Consequently, $M$ satisfies the finite exchange property.

\medskip
\noindent (ii) $M=N\oplus L$ where $N$ is square-free and and $L$
is $\mathcal{F}$-endomorphism-coinvariant.

\medskip
\noindent (iii) $M$ is a clean module.\end{enumerate}
\end{theor}

\noindent In particular, the above theorem applies to the case of the Enochs cotorsion pair, in which $\mathcal F$ is the class of flat modules and $\mathcal C$, the class of cotorsion modules.

\bigskip

{\bf Acknowledgments.} The authors would like to thank the referee for his/her careful reading and helpful comments. The authors would also like to thank Greg Marks for his comments on an earlier draft of this paper.

\bigskip

\bigskip

\end{document}